\documentclass[11pt,reqno]{amsart}
\usepackage{amsmath,amssymb,mathrsfs}
\usepackage{graphicx,cite}
\usepackage{color}
\usepackage{subfigure}
\usepackage{appendix}

\newtheorem{theorem}{Theorem}[section]

\newtheorem{lemma}[theorem]{Lemma}

 \newcommand{\et}{\zeta} 
\newcommand{\E}{\mathbb{E}} 
\newcommand{\ii}{{\rm i}}
\numberwithin{equation}{section}

\allowdisplaybreaks[4]

\begin{document}

\title[stability for inverse random source problems]{Stability for inverse random source problems of the polyharmonic wave equation}

\author{Peijun Li}
\address{LSEC, ICMSEC, Academy of Mathematics and Systems Science, Chinese Academy of Sciences, Beijing 100190, China, and School of Mathematical Sciences, University of Chinese Academy of Sciences, Beijing 100049, China}
\email{lipeijun@lsec.cc.ac.cn}

\author{Zhenqian Li}
\address{LSEC, ICMSEC, Academy of Mathematics and Systems Science, Chinese Academy of Sciences, Beijing 100190, China}
\email{lizhenqian2020@amss.ac.cn}

\author{Ying Liang}
\address{Department of Mathematics, Duke University, Durham, NC 27708, USA}
\email{ying.liang@duke.edu}




\begin{abstract}
This paper investigates stability estimates for inverse source problems in the stochastic polyharmonic wave equation, where the source is represented by white noise. The study examines the well-posedness of the direct problem and derives stability estimates for identifying the strength of the random source. Assuming a priori information of the regularity and support of the source strength, the H\"{o}lder stability is established in the absence of a potential. In the more challenging case where a potential is present, the logarithmic stability estimate is obtained by constructing specialized solutions to the polyharmonic wave equation.
\end{abstract}

\maketitle

\section{Introduction}

Inverse source problems, which aim to reconstruct unknown sources from observed wave field data, are fundamental to many applications, including optical and photoacoustic imaging, speech and sound source localization, as well as environmental monitoring and pollution detection. However, the inherent challenges in solving such problems stem from non-uniqueness, particularly due to non-radiating sources, which can produce identical wave fields outside the region of interest despite having different internal structures \cite{Bleistein77}. This ambiguity complicates the reconstruction, especially when boundary measurements are only available at a single frequency. Moreover, the inverse source problem is highly sensitive to noise, as small errors in the data can significantly distort the computed solutions. In response to these difficulties, research has shifted toward utilizing multi-frequency data to overcome the issues of uniqueness and stability. Previous work, notably by \cite{BLT2010jde}, demonstrated that incorporating multi-frequency measurements leads to unique and stable solutions in the context of the acoustic wave equation. Subsequent studies have extended these results, improving stability estimates for inverse source problems in more complex wave systems, including elastic and electromagnetic waves \cite{BLZ-JMPA20}. We refer to \cite{BLLT15} for a review of recent developments in theoretical and computational approaches to inverse scattering problems using multi-frequency data.

To address unpredictable environmental conditions, incomplete system information, and uncertainties arising from measurement noise, random parameters are introduced into mathematical models. This inclusion significantly increases the complexity of the inverse problem, transforming it from a deterministic framework into a more challenging stochastic one. The objective of the inverse random source problem is to recover statistical characteristics, including the mean and variance of the random parameters, to quantify the uncertainties associated with the source \cite{hohage2020ip}. The inverse random source problem for acoustic waves was first studied in \cite{devaney1979jmp}. Further advancements include the computational framework introduced in \cite{BCL16}, where random sources are modeled as white noise, a widely used uncorrelated stochastic process. Recently, \cite{Li2023stability} examined the stability of inverse source problems related to the three-dimensional Helmholtz equation, with the source represented by white noise. The study established H\"{o}lder- and logarithmic-type stability for homogeneous and inhomogeneous media, respectively, based on the correlation of Cauchy data from random wave fields.

Polyharmonic wave scattering problems have recently attracted growing interest due to their significant applications in engineering and materials science. This mathematical model plays a critical role in analyzing vibrational modes and dynamic responses in mechanical systems and composite materials, aiding the design of structures for various applications, such as aerospace and automotive engineering. For background and theoretical results on polyharmonic equations, we refer to \cite{GGS10} and the references therein. Inverse boundary value problems for the polyharmonic operator are studied in \cite{ST19, TV16, Krupchyk14}.

The aim of this work is to establish stability estimates for inverse random source problems in the context of polyharmonic wave equations. The inverse random source problem involves determining the strength of the source function using the correlation of measurements of random wave fields. We show that H\"{o}lder- and logarithmic-type stability can be achieved for the inverse random source problem for polyharmonic wave equations, with and without a potential, respectively, by utilizing correlation-based data at a single frequency and constructing special solutions to the polyharmonic wave equations. These results are unattainable for deterministic counterparts due to the presence of non-radiating sources. The use of complex geometric optics solutions to tackle inverse wave scattering problems can be found in \cite{hahner2001new} for deterministic inverse medium problems and in \cite{Li2023stability} for inverse random source problems for acoustic wave equations. In this work, we extend this approach to higher-order elliptic equations, simplifying the previous framework by removing the dependence of the constructed solutions on the wavenumber.

The structure of the paper is as follows: Section 2 addresses the well-posedness of the direct problem. Section 3 is dedicated to deriving stability estimates for the inverse source problems of stochastic polyharmonic wave equations, with and without a potential. Concluding remarks are provided in Section 4.

\section{The direct problem}

In this section, we introduce the necessary preliminaries for the model equation and examine the well-posedness of the associated direct problem.

Consider the stochastic polyharmonic wave equation
\begin{align}\label{eq:poly}
(-\Delta)^nu-k^{2n}u+qu=f\quad\text{in} ~ \mathbb R^3,
\end{align}
where $n\ge1$ is an integer,  $k>0$ is the wavenumber, and $q\in L^\infty(B_1)$ represents the potential function, with $B_1$ denoting the unit ball in $\mathbb{R}^3$. In particular, the Helmholtz equation and the biharmonic wave equation are two specific instances of the polyharmonic wave equation, corresponding to $n=1$ and $n=2$, respectively. The random source $f$ is assumed to take the form
\[
 f(x) = \sqrt{\sigma}(x) \dot{W}_x,
\]
where $\dot W_x$ is the spatial white noise, with $W_x$ being the Brownian motion defined on the complete probability space $(\Omega, \mathcal F, \mathbb P)$. The function $\sigma\geq 0$ is referred to as the strength of the random 
source, and it satisfies $\sigma\in H_0^s(B_1)$ with $s>3$ and $\sqrt{\sigma}\in H^{\frac{3}{2}+\delta}(B_1)$, where $0<\delta<\frac{1}{2}$. In addition, the Sommerfeld radiation condition is imposed to ensure the well-posedness of the stochastic polyharmonic wave equation (cf. \cite{LLWY24}):
\begin{equation}\label{eq:radia}
\lim_{r=|x|\to\infty}r\left(\partial_r u-{\rm i}k u\right)=0. 
\end{equation}

Given the random source $f$, the direct problem involves analyzing the well-posedness of \eqref{eq:poly}--\eqref{eq:radia}. In contrast, the inverse problem seeks to recover the strength $\sigma$ by utilizing the statistical properties of the wave field $u$, measured on the boundary of the ball $B_R\subset\mathbb R^3$, where $B_R$ has radius $R>1$.

To investigate the well-posedness of the direct problem, we begin by examining the properties of the Green's function. It can be verified using operator splitting that
\begin{equation}\label{os}
\left((-\Delta)^n-k^{2n}\right)^{-1}=\frac1{nk^{2n}}\sum_{j=0}^{n-1}\kappa_j^2(-\Delta-\kappa_j^2)^{-1},
\end{equation}
where $\kappa_j=ke^{{\rm i}\frac{j\pi}n}, j=0,1,\cdots,n-1$. Let $G$ denote the Green's function corresponding to the polyharmonic wave operator $(-\Delta)^n-k^{2n}$, which satisfies
\[
((-\Delta)^n-k^{2n})G(x,y,k)=\delta(x-y). 
\] 
Utilizing the operator splitting \eqref{os}, we obtain 
\begin{equation*}\label{eq:Green1}
G(x,y,k)=\frac1{nk^{2n}}\sum_{j=0}^{n-1}\kappa_j^2\Phi(x,y,\kappa_j),
\end{equation*}
where $\Phi$ is the Green's function of the Helmholtz equation in $\mathbb R^3$, given explicitly by 
\begin{equation*}
\Phi(x,y,\kappa_j)=\frac1{4\pi}\frac{e^{{\rm i}\kappa_j|x-y|}}{|x-y|}.
\end{equation*}

Define two integral operators
\begin{align*}
\mathcal H_k(\phi)(x):&=\int_{\mathbb R^3}G(x,y,k)\phi(y)dy,\\
\mathcal K_k(\phi)(x):&=\int_{\mathbb R^3}G(x,y,k)q(y)\phi(y)dy. 
\end{align*}
The following lemma provides estimates for the operators $\mathcal H_k$ and $\mathcal K_k$ with respect to the wavenumber $k$. The proof can be found in \cite{li24}.

\begin{lemma}\label{lm:H}
Suppose $D$ and $B$ are two bounded Lipschitz domains of $\mathbb{R}^3$. 
\begin{enumerate}
\item[(i)] The operator $\mathcal H_k: H^{-s_1}(D)\to H^{s_2}(B)$ is bounded, and it satisfies the estimate
\begin{align*}
\|\mathcal H_k\|_{\mathcal L(H^{-s_1}(D),H^{s_2}(B))}\lesssim k^{s-2n+1},
\end{align*}
where $s:=s_1+s_2\in[0,2n-1)$ and $s_1, s_2\ge0$. Furthermore, this boundedness extends to $\mathcal H_k: H^{-s}(D)\to L^\infty(B)$, with
\begin{align*}
\|\mathcal H_k\|_{\mathcal L(H^{-s}(D),L^\infty(B))}\lesssim k^{s-2n+\frac52+\epsilon}, 
\end{align*}
where $s\in[0,2n-1)$ and $\epsilon>0$.

\item[(ii)] For any $p, q > 1$ satisfying $\frac{1}{p}+\frac{1}{q}=1$, the operator $\mathcal H_k$ is compact from $W^{-\gamma,p}(D)$ to $W^{\gamma,q}(B)$, where $0<\gamma<\min\big\{\frac{2n-1}{2}, \frac{2n-1}{2}+3(\frac{1}{q}-\frac{1}{2})\big\}$.
\end{enumerate}
\end{lemma}

Hereafter, the notation $a \lesssim b$ denotes that there exists a generic positive constant $C$ such that $a\leq Cb$. Using the definition of the operator $\mathcal K_k$ and the properties of $\mathcal H_k$, we can establish the following lemma.

\begin{lemma}\label{lm:K}
Let $B\subset\mathbb R^3$ be a bounded domain. The operator $\mathcal K_k$ is bounded from $L^2(B)$ to $H^s(B)$ and from $L^2(B)$ to $L^\infty(B)$, satisfying
\begin{align*}
\|\mathcal K_k\|_{\mathcal L(L^2(B),H^s(B))}\lesssim k^{s-2n+1}, \quad 
\|\mathcal K_k\|_{\mathcal L(L^2(B),L^\infty(B))}\lesssim k^{-2n+\frac52+\epsilon}
\end{align*}
where $s\in[0,2n-1)$ and $\epsilon>0$. In addition, if $B$ satisfies the strong local Lipschitz condition, then $\mathcal K_k:H^n(B)\to H^n(B)$ is a compact operator.
\end{lemma}

\begin{proof}
According to the definition of $\mathcal K_k$ and Lemma \ref{lm:H}, we have
\begin{align*}
\|\mathcal K_k\phi\|_{H^s(B)}&=\|\mathcal H_k(q\phi)\|_{H^s(B)}\\
&\lesssim k^{s-2n+1}\|q\phi\|_{L^2(B)}\\
&\lesssim k^{s-2n+1}\|q\|_{L^\infty(B)}\|\phi\|_{L^2(B)}
\end{align*}
for any $\phi\in L^2(B)$, and
\begin{align*}
\|\mathcal K_k\psi\|_{L^\infty(B)}&=\|\mathcal H_k\|_{\mathcal L(L^2(B),L^\infty(B))}\|q\psi\|_{L^2(B)}\\
&\lesssim k^{-2n+\frac52+\epsilon}\|q\psi\|_{L^2(B)}\\
&\lesssim k^{-2n+\frac52+\epsilon}\|q\|_{L^\infty(B)}\|\psi\|_{L^2(B)}
\end{align*}
for any $\psi\in L^2(B)$. 
 
The compactness of the operator $\mathcal K_k:H^n(B)\to H^n(B)$ then follows from the compactness of the embedding $H^n(B)\hookrightarrow L^2(B)$, by the Rellich--Kondrachov theorem (cf. \cite[Theorem 6.3]{AF03}).
\end{proof}

Using these integral operators, the direct problem defined by \eqref{eq:poly}--\eqref{eq:radia} can be formulated as the Lippmann--Schwinger equation: 
\begin{equation}\label{eq:Lip}
u+\mathcal{K}_k u= \mathcal{H}_k f.
\end{equation}

\begin{theorem}\label{thm:exist_inhomo}
For sufficiently large wavenumber $k$, the Lippmann--Schwinger equation \eqref{eq:Lip} admits a unique solution $u\in H^n_{loc}(\mathbb R^3)$. 
\end{theorem}

\begin{proof}
Let $B\subset \mathbb{R}^3$ be any bounded domain with a locally Lipschitz boundary. By Lemma \ref{lm:H}, setting $s_1=n-1-\epsilon$ and $s_2=n$ for $\epsilon>0$, we have that the operator $\mathcal{H}_k:H^{-n+1+\epsilon}(B)\to H^{n}(B)$ is bounded. According to \cite{Veraar11},  the white noise  $\dot W_x \in H^{-\frac{3}{2}-\epsilon}(B)$ for $\epsilon>0$. By \cite[Corollary 2.1]{Tambaca01}, since $\sqrt{\sigma}\in H^{\frac{3}{2}+\delta}(B)$, it follows that the source $f= \sqrt{\sigma}(x) \dot{W}_x$ belongs to $H^{-2+\delta-\epsilon}(B)$. Therefore, for $n\geq 3$, taking $\epsilon<\frac{\delta}{2}$ leads to $f\in H^{-n+1+\epsilon}(B)$, and thus $\mathcal{H}_k f \in H^n(B)$. By Lemma \ref{lm:K}, the operator $\mathcal K_k: H^n(B)\to H^n(B)$ is compact. 
 
It remains to prove that $u+\mathcal{K}_k u=0$ has only the trivial solution in $H^n(B)$.  Recalling Lemma \ref{lm:K}, it is shown that $\|\mathcal K_k\|_{\mathcal L(L^2(B),H^s(B))}\lesssim k^{s-2n+1}$. If $u^*$ is a solution to $u^*+\mathcal{K}_k u^*=0$, we have the estimate $\Vert u^*\Vert_{H^n(B)}\leq \|\mathcal K_k\|_{\mathcal L(L^2(B), H^n(B))} \Vert u^*\Vert_{L^2(B)}\lesssim k^{-n+1}\Vert u^*\Vert_{L^2(B)}.$ Thus, $u^*=0$ for sufficiently large wavenumber $k$, which completes the proof.
\end{proof}

The following result demonstrates the well-posedness of the direct problem \eqref{eq:poly}--\eqref{eq:radia} in the distributional sense, whose proof is similar to that of in \cite[Theorem 3.2]{li24} and is omitted here.

\begin{theorem}\label{thm_weak}
For sufficiently large wavenumber $k$, the scattering problem admits a unique solution $u\in H^n_{loc}(\mathbb R^3)$ in the distribution sense.
\end{theorem}

The unique solution $u$ is a weak solution to $(-\Delta)^nu-k^{2n}u=0$ in the domain $B_{R+1}\setminus \overline{B_{\frac{1+R}{2}}}$ for $R>1$. By applying the interior regularity of weak solutions to higher order elliptic equations \cite[Theorem 11.1]{Taylor96}, we conclude that  $u\in H^{2n}_{loc}(B_{R+1}\setminus \overline{B_{\frac{1+R}{2}}}) $. Therefore, the quantities $\Delta^j u$ and $\partial_\nu \Delta^j u$ for $j=0, 1,2,...,n-1$ are well-defined when evaluated on the boundary $\partial B_R$. Accordingly, we can introduce the following correlation functions for the data acquired from the boundary $\partial B_R$:
\begin{align*}
F^1_{i,j} (x,y)& =\E[\Delta^i u(x) \Delta^j u(y)],\\
F^2_{i,j} (x,y)& =\E[\partial_\nu \Delta^i u(x) \Delta^j u(y)],\\
 F^3_{i,j} (x,y)& =\E[\Delta^i u(x) \partial_\nu \Delta^j u(y)],\\
 F^4_{i,j} (x,y)& =\E[  \partial_\nu\Delta^i u(x) \partial_\nu \Delta^j u(y)],
\end{align*}
Let
\begin{equation*}
 M = \max_{\alpha =1,2,3,4, (i,j)\in\Sigma}\left\{ \Vert F^\alpha_{i,j}\Vert_{L^2(\partial B_R\times \partial B_R)}\right\},
 \end{equation*}
 where
 \[
 \Sigma = \{(i,j)\in \mathbb{N}^2:0\leq i\leq j < n\}. 
 \]
The constant $M$ plays a crucial role in the stability estimates.

\section{The inverse problems}

In this section, we investigate the stability of the inverse source problem by estimating the Fourier coefficients of the source strength $\sigma$. Specifically, we analyze the low and high frequency coefficients separately. The cases with and without a potential are treated using different special solutions for the estimation of the low frequency components. This approach is inspired by previous stability estimates for the inverse random source problem in the Helmholtz equation \cite{Li2023stability}, as well as for the inverse medium scattering problem in a deterministic setting \cite{hahner2001new, isaev2013new}. By Theorem \ref{thm_weak}, the wavenumber $k$ is assumed to be sufficiently large to ensure that the direct problem has a unique solution.

The Fourier transform of the source strength $\sigma$ in the high frequency regime exhibits a decay behavior, as characterized by the Paley--Wiener--Schwartz theorem \cite{hormander03}.

\begin{lemma}\label{PWS}
For $\sigma \in H^{s}_0(B_1)$, let $\hat{\sigma}(\gamma)$ denote the Fourier transform of $\sigma$. There exists a positive constant $c_0$, depending only on $s$, such that
\begin{equation*}
|\hat{\sigma}(\gamma)|\leq c_0 (1+|\gamma|)^{-s}.
\end{equation*}
\end{lemma}

For $\zeta > 0$, it follows from Lemma \ref{PWS} that there exists a positive constant $c_1$, depending on $s$, such that
\begin{equation}\label{eqn:fourierhigh}
\int_{|\gamma|>\et} |\hat{\sigma}(\gamma)|d\gamma\leq c_1 \et^{3-s}.
\end{equation}

\subsection{Without a potential}

We begin by analyzing the stability of the inverse random source problem in the absence of a potential.

\begin{lemma}\label{lem:Ewhitehomo}
For two solutions $U_i\in H^{2n}\left(\overline{B_{\hat R}}\right)$ for $i=1,2$ of the equation $(-\Delta)^nu-k^{2n}u=0$ with $\hat R>R>1$, there exists a positive constant $c_2$, depending on $R$ and $\hat R$, such that 
\begin{equation*}\label{Ewhitehomo-s1}
|\E[\langle f, U_1\rangle \langle f, U_2\rangle]| \leq c_2 M \Vert U_1\Vert_{H^{2n}(B_{\hat R})} \Vert U_2\Vert_{H^{2n}(B_{\hat R})}. 
\end{equation*}
\end{lemma}

\begin{proof}
By employing the equation $(-\Delta)^n u-k^{2n}u=f$ and applying integration by parts, we obtain
\begin{align*}
 \langle f, U_i\rangle  &= \int_{B_R} f(x) U_i(x)dx\\
 &= \int_{B_R}\left((-\Delta)^n u(x) -k^{2n} u(x)\right) U_i(x)dx\\
 &=\int_{B_R} (-\Delta)^n u(x)U_i(x)dx -k^{2n}\int_{B_R} u(x) U_i(x)dx\\
 &=\int_{B_R} u(x) (-\Delta)^n U_i(x)dx -k^{2n}\int_{B_R} u(x)  U_i(x)dx \\
 &\quad +(-1)^n\sum_{j=1}^n\left( \int_{\partial B_R} \Delta^{n-j} U_i(x)\partial_\nu \Delta^{j-1} u(x)ds(x) -\int_{\partial B_R} \Delta^{j-1} u(x)\partial_\nu\Delta^{n-j} U_i(x)ds\right)\\
&=(-1)^n\sum_{j=1}^n\left( \int_{\partial B_R} \Delta^{n-j} U_i(x)\partial_\nu \Delta^{j-1} u(x)ds(x) -\int_{\partial B_R} \Delta^{j-1} u(x)\partial_\nu\Delta^{n-j} U_i(x)ds\right), 
 \end{align*}
 where we have utilized $(-\Delta)^nU_i-k^{2n}U_i=0$ for $i=1, 2$. From straightforward calculations, we derive
\begin{align}
&\mathbb{E}\left[\langle f, U_1\rangle \langle f, U_2\rangle\right]  \nonumber\\
&= \int_{\partial B_R} \int_{\partial B_R}\mathbb{E}\bigg[\sum_{j=1}^n\left( \Delta^{n-j} U_1(x)\partial_\nu \Delta^{j-1} u(x)ds(x) - \Delta^{j-1} u(x)\partial_\nu\Delta^{n-j} U_1(x)\right)\nonumber\\
&\qquad \times \sum_{j=1}^n\left(  \Delta^{n-j} U_2(y)\partial_\nu \Delta^{j-1} u(y) - \Delta^{j-1} u(y)\partial_\nu\Delta^{n-j} U_2(y)\right) \bigg]ds(x)ds(y)\nonumber\\
&\lesssim  M\Vert U_1\Vert_{H^{2n}(B_{\hat{R}})}\Vert U_2\Vert_{H^{2n}(B_{\hat{R}})},\nonumber 
\end{align}
which completes the proof. 
\end{proof}

Next, we employ special solutions to the polyharmonic equation, in the absence of a potential, to derive an estimate for the Fourier coefficients $\hat{\sigma}(\gamma)$ for $|\gamma| < 2k$.

\begin{lemma}\label{lem:Fourierlow_white_homo}
For all $\gamma\in\mathbb{R}^3$ with $|\gamma|< 2k$, there exists a constant $c_3>0$, which depends on $k$, $R$, and $\hat R$, such that the following estimate holds:
\[
|\hat{\sigma}(\gamma)|\leq c_3 M.
\]
\end{lemma}

\begin{proof}
We construct two special solutions to the polyharmonic equation $(-\Delta)^{n} u -k^{2n} u =0$. For a fixed multi-index $\gamma\in \mathbb{R}^3$, we choose a unit vector $d\in\mathbb{R}^3$ satisfying $ d\cdot \gamma = 0$, and define two real vectors 
\begin{align*}
\xi^{(1)} = -\frac{1}{2}\gamma  +\Bigl(k^2-\frac{|\gamma|^2}{4}\Bigr)^{1/2} d,\quad
\xi^{(2)} =-\frac{1}{2}\gamma -\Bigl(k^2-\frac{|\gamma|^2}{4}\Bigr)^{1/2} d.
\end{align*}
Clearly, we have $\xi^{(1)}+\xi^{(2)} = -\gamma$, $\xi^{(1)}\cdot \xi^{(1)}= \xi^{(2)}\cdot \xi^{(2)} = k^2$. Consider two functions 
\begin{eqnarray*}
U_i(x,\xi^{(i)}) =e^{\ii  \xi^{(i)}\cdot x},\quad i=1,2.
\end{eqnarray*}
It can be verified that 
\[
((-\Delta)^{n}  -k^{2n}  ) U_i = 0.
\]
By the It\^{o} isometry, we have
 \begin{align*}
\E[\langle f, U_1\rangle \langle f, U_2\rangle]  &= \int_{\mathbb{R}^3} \int_{\mathbb{R}^3} U_1(x) U_2(y) \E [f(x)f(y)]dxdy\\
&=     \int_{\mathbb{R}^3} \sigma(x) e^{-\ii\gamma\cdot x}dx=(2\pi)^3\hat{\sigma}(\gamma).
\end{align*}
Using Lemma \ref{lem:Ewhitehomo}, we establish the existence of a positive constant $c_3$, depending on $k$, $R$, and $\hat R$, such that
\begin{equation}\label{eq:estlow_homo}
 |\hat{\sigma}(\gamma) |\leq c_2 (2\pi)^{-3} M \Vert U_1\Vert_{H^{2n}(B_{\hat R})} \Vert U_2\Vert_{H^{2n}(B_{\hat R})} \leq c_3 M,
\end{equation}
which completes the proof.
\end{proof}

Using the previously presented lemmas, we establish a stability estimate for the source strength $\sigma$ in the context of the polyharmonic wave equation without a potential. The proof closely follows the method outlined in \cite[Theorem 1.1]{Li2023stability} and is briefly summarized here for completeness.

\begin{theorem}\label{thm:white_homo}
There exist a positive constant $k_0$, depending on $s$, $M$, and $R$, and a positive constant $C_1$, depending on  $s$, $k$, and $R$, such that if  $k>k_0$,  the following estimate holds:
\begin{equation*}
\Vert \sigma\Vert_{L^\infty(B_1)}\leq  C_1 M^{1-\frac{3}{s}}. 
\end{equation*}
\end{theorem}

\begin{proof}
From Lemma \ref{lem:Fourierlow_white_homo} and \eqref{eqn:fourierhigh}, it follows that for $\zeta \leq 2k$,
\begin{align*}
\Vert \sigma\Vert_{L^\infty(B_1)} &\leq \sup_{x\in B_1}\bigg|\int_{\mathbb{R}^3} e^{\ii \gamma\cdot x} \hat{\sigma}(\gamma)d\gamma\bigg|\\
&\leq \int_{|\gamma|\leq \zeta} |\hat{\sigma}(\gamma)|d\gamma+ \int_{|\gamma|>\zeta} |\hat{\sigma}(\gamma)|d\gamma\leq \frac{4\pi \zeta^3}{3} c_3 M + c_1\zeta^{3-s}. 
\end{align*}
Recall that, by Lemma \ref{lem:Fourierlow_white_homo} and \eqref{eq:estlow_homo}, there exists a constant $\bar{c}_3$ independent of $k$ such that $c_3 = \bar{c}_3 k^{4n}$. Taking $k_0 =\left(\frac{3c_1}{2^{s+2}\pi \bar{c}_3 M}\right)^{\frac{1}{s+4n}}$ and assuming  $k>k_0$,  we observe that $\left(\frac{3c_1}{4\pi c_3 M}\right)^{\frac{1}{s}}\leq 2k$.  Substituting $\zeta = \left(\frac{3c_1}{4\pi c_3 M}\right)^{\frac{1}{s}}$ into the above estimates, we establish the existence of a positive constant $C_1$ such that
\begin{equation*}
\Vert \sigma\Vert_{L^\infty(B_1)}\leq C_1 M^{1-\frac{3}{s}},
\end{equation*}
thereby completing the proof.
\end{proof}

It is worth noting that another form of stability can be deduced. Specifically, we have
\begin{align*}
\Vert \sigma\Vert_{L^\infty(B_1)} &\leq \int_{|\gamma|\leq \zeta} |\hat{\sigma}(\gamma)|d\gamma+ \int_{|\gamma|>\zeta} |\hat{\sigma}(\gamma)|d\gamma\leq \frac{4\pi \zeta^3}{3} c_3 M + c_1\zeta^{3-s}. 
\end{align*}
By Lemma \ref{lem:Fourierlow_white_homo}, there exists a constant $\bar{c}_3$, independent of $k$, such that $c_3 = \bar{c}_3 k^{4n}$. Taking $\xi = k^{\frac{1}{s}}$, we obtain 
\begin{align*}
\Vert \sigma\Vert_{L^\infty(B_1)} &\leq \frac{4\pi}{3} \bar{c}_3 k^{4n+\frac{3}{s}} M + c_1k^{\frac{3}{s}-1}\\
&\leq k^{\frac{s}{3}}(\frac{4\pi}{3} \bar{c}_3 k^{4n} M + c_1k^{-1}). 
\end{align*}
Letting $k_0 =\left(\frac{3c_1}{4\pi \bar{c}_3 M}\right)^{\frac{1}{1+4n}}$ and assuming $k>k_0$,  we observe that $\frac{4\pi}{3} \bar{c}_3 k^{4n} M>c_1k^{-1}$. Thus, we establish the existence of a positive constant $C_1$ such that the following Lipschitz stability is achieved:
\begin{equation*}
\Vert \sigma\Vert_{L^\infty(B_1)}\leq C_1 M. 
\end{equation*}

\subsection{With a potential}

We now investigate the stability estimates for the polyharmonic equation with a potential term. We utilize special solutions to the equation
\begin{equation*}
(-\Delta)^n u- k^{2n} u+q u = 0,
\end{equation*}
which are known as complex geometric optics (CGO) solutions and possess the structure described in the following lemma. The proof is similar to that presented in \cite[Proposition 2.4]{Krupchyk14} and is included here for completeness.

\begin{lemma}\label{lem:CGO}
 For all $\xi\in\mathbb{C}^3$ satisfying $\xi\cdot\xi=0$ and $|\Im (\xi)|\geq c_4$, where $c_4$ is a positive constant depending on $k$ and $\hat{R}$, there exists a solution  $u(x,\xi)\in H^{2n}(B_{\hat{R}})$ to the equation $
 (-\Delta)^n u-k^{2n} u+q u = 0$. Moreover, the solution can be expressed in the form 
 \[
u(x,\xi) =e^{\ii x\cdot \xi}(1+v(x,\xi)),
\]
where $v(x,\xi)$ satisfies
\[
\Vert v(\cdot,\xi)\Vert_{L^2(B_{\hat{R}})} \leq \frac{c_5}{|\Im(\xi)|^n}, 
\]
with $c_5$ being a positive constant. 
\end{lemma}

\begin{proof}
We adopt the method of Carleman estimates to construct a solution for the polyharmonic wave equation of the following form:
\[
u(x,\xi) =e^{\ii x\cdot \xi}(1+v(x,\xi)),
\]
where $\xi\in\mathbb{C}^3$ and $\xi\cdot\xi =0$. We first introduce the Carleman estimate for the semi-classical polyharmonic operator $(-h^2\Delta)^n$, where $h>0$ is a real constant. Consider the pseudo-differential operator
\[
P_\varphi = e^{\frac{\varphi}{h}}(-h^2\Delta)e^{-\frac{\varphi}{h}}, 
\]
which has the principal symbol
\[
p_\varphi(x,\xi) = \xi^2+2\ii \nabla\varphi\cdot \xi-|\nabla\varphi|^2. 
\]
Here, $\varphi(x)=\alpha\cdot x$ denotes the linear weight, with $\alpha\in \mathbb{R}^3$ and $|\alpha|=1$. The Carleman estimate, as proven in \cite{Kenig2007calderon}, states that for all $u\in\mathbb{C}_0^\infty(\Omega)$,
\begin{equation*}
\Vert  e^{\frac{\varphi}{h}}(-h^2\Delta)e^{-\frac{\varphi}{h}} u\Vert_{H^s_{scl}}\geq \frac{h}{C_{s,\Omega}}\Vert u\Vert_{H_{scl}^{s+1}},
\end{equation*}
where $\Vert\cdot\Vert^s_{scl}$ denotes the semi-classical norm defined by $\Vert u\Vert_{scl}^s = \Vert \langle hD\rangle^s u\Vert_{L^2}$ with $\langle\xi\rangle=(1+|\xi|^2)^{1/2}$. By iterating this estimate $n$ times, we obtain the estimate for the polyharmonic operator:
\begin{equation*}
\Vert  e^{\frac{\varphi}{h}}(-h^2\Delta)^ne^{-\frac{\varphi}{h}} u\Vert_{H^s_{scl}}\geq \frac{h^n}{C_{s,\Omega}}\Vert u\Vert_{H_{scl}^{s+1}}.
\end{equation*}

Next we introduce the perturbation $-k^{2n}+q$. Let $\mathcal{L}_q: =  (-\Delta)^n - k^{2n} +q$ and define
\[
\mathcal{L}_\varphi = e^{\frac{\varphi}{h}}h^{2n} \mathcal{L}_q e^{-\frac{\varphi}{h}}.
\]
For $0<h\leq (\Vert -k^{2n}+q\Vert^{-1}_{L^\infty})^{\frac{1}{n}}$, we have for all $u\in C_0^\infty(\Omega)$,
\begin{equation}\label{appen:estphi}
\Vert \mathcal{L}_\varphi  u \Vert_{H^s_{scl}}\geq \frac{h^n}{C_{s,\Omega,q}}\Vert u\Vert_{H_{scl}^{s+1}}.
\end{equation}
Following the proof of \cite[Proposition 2.3]{Krupchyk14}, it can be shown using \eqref{appen:estphi} that  for any $v \in L^2(B_{\hat{R}})$, there exits a solution $u \in H^1(B_{\hat{R}})$ of the equation
\[
L_\varphi u = v\quad \text{in} \ B_{\hat{R}},
\]
which satisfies
\[
\Vert u\Vert_{H_{scl}^1}\leq  \frac{C}{h^n}\Vert v\Vert_{L^2}.
\]

We now construct the solution to the equation $\mathcal{L}_q u =0$ in the form 
\begin{equation}\label{appen:cgoform}
u= e^{\ii x\cdot\xi}(1+ v(x,\xi)),
\end{equation}
where $\xi\cdot\xi=0$ and $|\Re \xi|=|\Im \xi| = \frac{1}{h}$. Consider
\[
e^{-\ii x\cdot\xi} h^{2n}  \mathcal{L}_q  e^{\ii x\cdot\xi} = (-h^2\Delta -2\ii \xi\cdot\nabla)^n+h^{2n} (-k^{2n}+q).
\]
Note that 
\[
e^{-\ii x\cdot\xi} h^{2n}  \mathcal{L}_q  e^{\ii x\cdot\xi} 1 = h^{2n} (-k^{2n}+q).
\]
If the function $u$ in the form \eqref{appen:cgoform} satisfies the polyharmonic equation, then the remainder $v$ satisfies
\[
e^{-\ii x\cdot\xi} h^{2n}  \mathcal{L}_q  e^{\ii x\cdot\xi} v = -e^{-\ii x\cdot\xi} h^{2n}  \mathcal{L}_q  e^{\ii x\cdot\xi} 1,
\]
which leads to 
\begin{equation}\label{ap:resi}
e^{x\cdot\Im \xi} h^{2n}  \mathcal{L}_q  e^{-x\cdot\Im\xi}  e^{\ii x\cdot\Re\xi} v  = -e^{\ii x\cdot\Re \xi}  e^{-\ii x\cdot \xi} h^{2n}  \mathcal{L}_q  e^{\ii x\cdot\xi} 1.
\end{equation}

Note that $e^{x\cdot\Im \xi} h^{2n}  \mathcal{L}_q  e^{-x\cdot\Im\xi} =\mathcal{L}_\varphi$ with $\varphi(x) = h\Im\xi\cdot x$. Thus, the equation \eqref{ap:resi} is solvable, and the solution $v$ satisfies
\begin{align*}
\Vert e^{\ii x\cdot\Re\xi} v\Vert_{H_{scl}^1}&\leq \frac{C}{h^n}\Vert e^{\ii x\cdot\Re \xi}  e^{-\ii x\cdot \xi} h^{2n}  \mathcal{L}_q  e^{ix\cdot\xi} 1\Vert_{L^2}\\
&\leq C h^n (k^{2n} + 1),
\end{align*}
where $C$ is a generic positive constant. Noting that 
\[
\Vert v\Vert_{H_{scl}^1}^2 = \Vert v\Vert_{L^2}^2+ \Vert h\nabla v\Vert_{L^2}^2,
\]
we obtain
\[
\Vert e^{\ii x\cdot\Re\xi} v\Vert_{L^2(B_{\hat{R}})} = \Vert v\Vert_{L^2(B_{\hat{R}})}\leq C h^n {(k^{2n} + 1)}.
\]
Since $u$ is the  solution to the polyharmonic wave equation $(-\Delta)^n u= k^{2n} u -qu$, we conclude that $u\in H^{2n}(B_{\hat{R}})$ based on the interior regularity estimate (cf. \cite[Theorem 6.29]{grubb2008distributions}).
\end{proof}

The following result establishes a connection between the internal distribution of the random source and the measurements of the wave field $u$ on $\partial B_R$ for the wave equation with a potential term. The proof is similar to that of Lemma \ref{lem:Ewhitehomo} and is based on integration by parts; therefore, it is omitted.

\begin{lemma}\label{lem:Ewhitehomo_potential}
For two solutions $U_i\in H^{2n}\left(\overline{B_{\hat R}}\right)$, $i=1,2$, of the equation $(-\Delta)^nu-k^{2n}u+qu=f$ with $\hat R>R>1$, there exists a positive constant $c_6$, depending on $R$ and $\hat R$, such that 
\begin{equation}\label{Ewhitehomo-s2}
|\E[\langle f, U_1\rangle \langle f, U_2\rangle]| \leq c_6 M \Vert U_1\Vert_{H^{2n}(B_{\hat R})} \Vert U_2\Vert_{H^{2n}(B_{\hat R})}. 
\end{equation}
\end{lemma}

We can now present estimates for the Fourier coefficients $\hat{\sigma}(\gamma)$ associated with the low frequency modes, utilizing the CGO solutions.

\begin{lemma}\label{lem:Fourierlow}
Assume that $\zeta\geq 2$. Let $t_0 =\sqrt{c_4^2+ \zeta^2}$, where $c_4$ is defined in Lemma \ref{lem:CGO} for $\hat R= 2R$. Then, there exists a positive constant $c_7$ depending on $k$ and $R$, such that for all $\gamma\in\mathbb{R}^3$ with $|\gamma|< \zeta$ and for all $t>t_0$, the following estimate holds:
\[
|\hat{\sigma}(\gamma)|\leq c_7\left( M e^{4Rt}+\frac{\Vert \sigma\Vert_{L^\infty(B_1)} }{t^n}\right).
\]
\end{lemma}

\begin{proof}
For a fixed multi-index $\gamma\in \mathbb{R}^3$, we choose two unit real-valued vectors $d_1$ and $d_2$ such that $d_1\cdot d_2 = d_1\cdot \gamma= d_2\cdot\gamma = 0$. Since $|\gamma|< \zeta$, we define two complex-valued vectors for $t>t_0$:
\begin{align*}
\xi_t^{(1)} &= -\frac{1}{2}\gamma + \ii t d_1 +\left(-\frac{|\gamma|^2}{4}+t^2\right)^{1/2} d_2\in \mathbb{C}^3,\\
\xi_t^{(2)} &=-\frac{1}{2}\gamma - \ii t d_1 -\left(-\frac{|\gamma|^2}{4}+t^2\right)^{1/2} d_2\in \mathbb{C}^3.
\end{align*}
It is straightforward to verify that $\xi_t^{(1)}+\xi_t^{(2)} = -\gamma$, $\xi_t^{(1)}\cdot \xi_t^{(1)}= \xi_t^{(2)}\cdot \xi_t^{(2)} = 0$. Using Lemma \ref{lem:CGO}, we construct two CGO solutions 
\begin{eqnarray*}
 U_i(x,\xi_t^{(i)}) = e^{\ii  \xi_t^{(i)}\cdot x}(1+v_i(x,\xi_t^{(i)})),\quad i=1, 2,
\end{eqnarray*}
which satisfy $(-\Delta)^{2n} U -k^{2n} U+ qU= 0$ in $B_{2R}$, with $\Vert v_i\Vert_{L^2(B_{2R})}\leq \frac{c_5 }{|\Im \xi_t^{(i)}|^n}$. For simplicity, we denote $U_1(x,\xi_t^{(1)})$ and $U_2(x,\xi_t^{(2)})$ as $U_1(x)$ and $U_2(x)$, respectively. From formulation of $U_i(x)$, we have 
\begin{equation*}
U_1(x)U_2(x)  = e^{-\ii\gamma\cdot x}(1+p(x,t)),
\end{equation*}
where
\[
p(x,t) = v_1(x,\xi_t^{(1)}) + v_2(x,\xi_t^{(2)}) + v_1(x,\xi_t^{(1)})v_2(x,\xi_t^{(2)}).
\]
By applying the Cauchy--Schwarz inequality, we obtain
\begin{align}
\Vert p(\cdot, t)\Vert_{L^1(B_1)} &\lesssim  \Vert v_1\Vert_{L^2(B_1)} +\Vert v_2\Vert_{L^2(B_1)}  +\Vert v_1\Vert_{L^2(B_1)}  \Vert v_2\Vert_{L^2(B_1)}  \nonumber \\
&\lesssim  \frac{c_5 }{|\Im \xi_t^{(1)}|^n} + \frac{c_5 }{|\Im \xi_t^{(2)}|^n} + \frac{c_5^2 }{|\Im \xi_t^{(1)}|^n |\Im \xi_t^{(2)}|^n} \lesssim \frac{c_5}{t^n}.\label{eq:estp}
\end{align}
 
Using the  It\^{o} isometry, we obtain
\begin{align*}
\E[\langle f, U_1\rangle \langle f, U_2\rangle]  &= \E\left[\int_{\mathbb{R}^3} \int_{\mathbb{R}^3} \sqrt{\sigma}(x)U_1(x) \sqrt{\sigma}(y) U_2(y)  \dot{W} (x) \dot {W}(y)dxdy\right] \nonumber\\
&=\int_{B_1}  \sigma(x)U_1(x) U_2(x)dx \nonumber\\\
&= \int_{B_1}  \sigma(x) e^{-\ii\gamma\cdot x}(1+p(x,t))dx  \nonumber\\\
&=(2\pi)^3\hat{\sigma}(\gamma) +  \int_{B_1}  \sigma(x) e^{-\ii\gamma\cdot x}p(x,t)dx,
\end{align*}
which, together with \eqref{Ewhitehomo-s2} and  \eqref{eq:estp}, yields  
\begin{align}
|\hat{\sigma}(\gamma)| &\leq (2\pi)^{-3}\left( | \E[\langle f, U_1\rangle \langle f, U_2\rangle]| + \left|\int_{B_1}  \sigma(x) e^{-\ii\gamma\cdot x}p(x,t))dx\right|\right) \nonumber \\
&\lesssim M \Vert U_1\Vert_{H^{2n}(B_{\frac{3}{2} R})} \Vert U_2\Vert_{H^{2n}(B_{\frac{3}{2} R})} +  \frac{k^{2n}}{t^n}\Vert \sigma\Vert_{L^\infty(B_1)}.\label{eq:est_low_b1}
\end{align}
Since $U_i$ satisfies $(-\Delta)^{n} U_i -k^{2n} U_i+qU_i=0$ for $i=1,2$, we derive from  the interior regularity  \cite[Theorem 6.29]{grubb2008distributions} and the H\"{o}lder inequality that 
\begin{align*}
\Vert  U_i\Vert_{H^{2n}(B_{\frac{3}{2} R})}&= \Vert e^{\ii \xi_t^{(i)}\cdot x}(1+v_i(x,\xi_t^{(i)}))\Vert_{H^{2n}(B_{\frac{3}{2} R})}\\
&\lesssim k^{2n}\Vert e^{\ii \xi_t^{(i)}\cdot x}(1+v_i(x,\xi_t^{(i)}))\Vert_{L^2(B_{2R})}\\
&\lesssim k^{2n} e^{2Rt}.
\end{align*}
Together with \eqref{eq:est_low_b1}, this completes the proof.
\end{proof}

The following result presents a stability estimate for the inverse source problem associated with the polyharmonic wave equation that includes a potential term.

\begin{theorem}\label{thm:white_inhomo}
 There exists a positive constant $C_2$, depending on $s$, $k$, and $R$, such that the following inequality holds:
\begin{equation*} 
\Vert \sigma\Vert_{L^\infty(B_1)}\leq C_2\left(\ln(3+M^{-1})\right)^{n(1-\frac{s}{3})}.
\end{equation*} 
\end{theorem}

\begin{proof}
Note that
\begin{align*}
\Vert \sigma\Vert_{L^\infty(B_1)} &\leq \sup_{x\in B_1}\left|\int_{\mathbb{R}^3} e^{\ii \gamma\cdot x} \hat{\sigma}(\gamma)d\gamma\right|\\
&\leq \int_{|\gamma|\leq \et} |\hat{\sigma}(\gamma)|d\gamma+ \int_{|\gamma|>\et} |\hat{\sigma}(\gamma)|d\gamma:=I_1(\zeta) + I_2(\zeta).
\end{align*}
Using Lemma \ref{lem:Fourierlow}, for $t>t_0$, we obtain
\begin{equation*}
I_1(\zeta)\leq \frac{4\pi \zeta^3}{3}  c_7\left(  M e^{4Rt}+\frac{\Vert \sigma\Vert_{L^\infty(B_1)} }{t^n}\right).
\end{equation*}
Utilizing \eqref{eqn:fourierhigh} and letting $c_8=\frac{4\pi}{3}c_7$, we get
\begin{eqnarray*}
\Vert \sigma\Vert_{L^\infty(B_1)} \leq c_8\et^3 \left(  M e^{4Rt}+\frac{\Vert \sigma\Vert_{L^\infty(B_1)} }{t^n}\right)+c_1\et^{3-s}.
\end{eqnarray*}
substituting $\et = (\frac{t^n}{2c_8})^{\frac{1}{3}}$ into the above estimate leads to 
\begin{eqnarray*}
\frac{1}{2}\Vert \sigma\Vert_{L^\infty(B_1)} \leq \frac{1}{2}M t^n e^{4Rt}+c_1 (2c_8)^{\frac{s-3}{3}} t^{-\frac{n(s-3)}{3}}.
\end{eqnarray*}
Let $t = (1-\tau)\frac{\ln(3+M^{-1})}{4R}$, where $\tau\in(0,1)$ and $M$ is sufficiently small such that $t> t_0$. Then, we have
\begin{align*}
\Vert \sigma\Vert_{L^\infty(B_1)} &\leq M(3+M^{-1})^{1-\tau} \left(\frac{\ln(3+M^{-1})}{4R}\right)^n\\
&\quad +2c_1 (2c_8)^{\frac{s-3}{3}}\left((1-\tau)\frac{\ln(3+M^{-1})}{4R}\right)^{-\frac{n(s-3)}{3}}\\
&\leq  (3M+1)^{1-\tau} M^\tau \left(\frac{\ln(3+M^{-1})}{4R}\right)^n\\
&\quad +2c_1 \left(\frac{(4R)^nc_8}{(1-\tau)^n}\right)^{\frac{s-3}{3}}\left(\ln(3+M^{-1})\right)^{-\frac{n(s-3)}{3}}.
\end{align*}
Assuming that $M$ is less than a sufficiently small constant $\delta_1>0$, which depends on $s$, $k$, $R$, and $\tau$, we establish the existence of a positive constant $C_2$ such that 
\begin{align*}
\Vert \sigma\Vert_{L^\infty(B_1)} &\leq  4c_1 \left(\frac{(4R)^nc_8}{(1-\tau)^n}\right)^{\frac{s-3}{3}}\left(\ln(3+M^{-1})\right)^{-\frac{n(s-3)}{3}}\\
&\leq C_2 (\ln(3+M^{-1}))^{n(1-\frac{s}{3})},
\end{align*}
which completes the proof.
\end{proof}

\section{Conclusion}\label{sec:4}

In this paper, we have studied the stability of inverse source problems associated with the stochastic polyharmonic wave equation with white noise as the driving force. We have demonstrated that the direct problem has a unique distributional solution and established H\"{o}lder- and logarithmic-type stability for the inverse problems, in the presence and absence of potential, respectively. We plan to present further developments in future work, where we will address the challenges posed by generalized random fields and extend our approach to a broader class of wave equations.

\end{document}